\documentclass{amsart}
\usepackage{amssymb,multicol}

\usepackage[all]{xy}

\usepackage{hyperref}
\hypersetup{
    colorlinks=true,
    linkcolor=blue,
    filecolor=blue,      
    urlcolor=cyan,
    citecolor=blue,
}

\usepackage{tikz}
\usetikzlibrary{tqft}
\usetikzlibrary{arrows.meta, decorations.pathreplacing}

\newtheorem{thm}[subsection]{Theorem}
\newtheorem{cor}[subsection]{Corollary}
\newtheorem{lem}[subsection]{Lemma}
\newtheorem{prop}[subsection]{Proposition}

\newcommand{\sta}{\stackrel}
\newcommand{\ci}{{\circ}}
\newcommand{\op}{\operatorname}
\newcommand {\CC} {\mathbb{C}}
\newcommand{\ben}{\begin{equation}}
\newcommand{\een}{\end{equation}}
\newcommand{\ZZ}{\mathbb{Z}}
\newcommand{\QQ}{\mathbb{Q}}
\newcommand{\K}{Kunyavski\u\i }
\newcommand{\defeq}{\stackrel{\text{def}}{=}}

\DeclareMathOperator{\GL}{GL}
\DeclareMathOperator{\Tor}{Tor}

\begin{document}

\title[Topological perspectives on some Bogomolov multipliers]{Topological perspectives on the vanishing of some Bogomolov multipliers}

\author{Eric Samperton}
\address{Purdue University, West Lafayette, Indiana, USA}
\email{eric@purdue.edu}
\thanks{ES supported in part by NSF CCF 2330130}

\author{Carlos Segovia}
\address{SECIHTI-UNAM-Oaxaca,  M\'{e}xico}
\email{csegovia@im.unam.mx}
\thanks{CS supported by Investigadores por M\'exico SECIHTI.}


\date{August 14, 2026}

\begin{abstract}
Since the 1980s, the Bogomolov multiplier of a finite group has been known to obstruct rationality in complex algebraic geometry, and more recently it is understood to be responsible for any torsion in the oriented and stable unitary 2-dimensional $G$-equivariant bordism groups $\Omega_2^{SO,G}$ and $\Omega_2^{U,G}$.
In this note, as a small step toward building a bridge between these two far-flung roles, we discuss  the vanishing of Bogomolov multipliers of two specific families of finite groups.
First, we revisit Kunyavski\u\i's result that the Bogomolov multipliers of all finite simple groups vanish, taking inspiration from the low-dimensional topological interpretation of the Ore conjecture.
Second, in lieu of arguments in complex birational geometry (such as the hard direction of the Chevalley–Shephard–Todd theorem), we combine cut-and-paste combinatorial-topological techniques with elementary calculations of Ihara-Yokonuma to show that all finite Coxeter groups have vanishing Bogomolov multiplier.
\end{abstract}

\maketitle

\section*{Introduction}
\label{intro}
The Noether problem over a field $k$, first formulated in 1918, asks: given a finite group $G$ and a (faithful) representation of $G$ on a vector space $V$ over $k$, is the quotient algebraic variety $V/G$ rational?
Noether's motivation was the case $k=\mathbb{Q}$, since a positive solution for $G$ here implies a positive solution for $G$ in the inverse Galois problem \cite{Noether}.
However, Swan later showed that for $G=\ZZ/47\ZZ$ the Noether problem over $\mathbb{Q}$ does not have a positive solution \cite{Swan}, even though $G$ (being abelian) \emph{does} have a positive solution to the inverse Galois problem.
This instigated further inquiry, leading to further negative solutions to Noether's problem over other fields $k$, and culminating in Saltman's examples of finite groups $G$ admitting representations on \emph{complex}---that is, $k=\CC$---vector spaces $V$ such that $V/G$ is not rational \cite{Sal84}.

To this end, Saltman introduced the so-called unramified Brauer group 
\[Br_v(\CC(V/G)) \le Br(\CC(V/G))\]
of the variety $V/G$.
He showed that $Br_v(\CC(V/G))$, like $Br(\CC(V/G))$,  is a birational invariant, and that there exist groups with representations where $Br_v(V/G) \ne 0$.
Shortly thereafter, Bogomolov showed in \cite{Bo87} that the somewhat unwieldy definition of $Br_v(V/G)$ could be identified with
\[ B_0(G) \defeq \{ \omega \in H^2(G, \mathbb{Q}/\ZZ) \mid \omega|_A = 0 \in H^2(A, \mathbb{Q}/\ZZ), \ \forall A \le G \text{ abelian}\}.\]
Being a subgroup of the (Pontryagin dual of the) Schur multiplier $H^2(G, \mathbb{Q}/\ZZ)$, \K\ named $Br_v(\CC(V/G)) = B_0(G)$ the \emph{Bogomolov multiplier} of $G$ \cite{K}.
We comment more on this terminology in a moment.

Just as $H^2(G, \QQ/\ZZ)$ is Pontryagin dual to the Schur multiplier \[ M(G) \defeq H_2(G, \ZZ),\]
Moravec \cite{Mo12} identifies $B_0(G)$ with the Pontryagin dual of
\[\tilde{B}_0(G) \defeq M(G)/M_0(G), \]
where $M_0(G)$ is defined as the subgroup of $M(G)$ generated by all elements of the form $x \wedge y$ (via the natural identification of $M(G)$ with the kernel of $G \wedge G \to [G,G]$ \cite{Mi52}; see Section 2 for more discussion).
While finite abelian groups are indeed isomorphic to their Pontryagin duals, as far as we are aware there is no \emph{natural} isomorphism between $B_0(G)$ and $\tilde{B}_0(G)$.
So, given that $M(G) = H_2(G, \ZZ)$ is ``the" Schur multiplier, in this work, our preferred definition of ``the" Bogomolov multiplier is
\[ B(G) \defeq \tilde{B}_0(G).\]
As justification for this preference, we note that Moravec shows that $B(G)$ satisfies properties analogous to the fact that $M(G)$ is the kernel of any minimal order stem extension of $G$ (see Section 1) \cite{Mo12}.

Note that $x \wedge y \mapsto 1 \in [G,G]$ if and only if $x$ and $y$ commute, and pairs of commuting elements of $G$ are in bijection with homomomorphisms
\[ f: \pi_1(S^1 \times S^1) = \ZZ \times \ZZ \to G.\]
In fact, the generators $x \wedge y$ of $M_0(G)$ are precisely the same thing as homology classes in $M(G) = H_2(G, \ZZ) = H_2(BG, \ZZ)$ that are \emph{toral}, meaning they are representable as
\[ f_*[S^1 \times S^1]\]
for some choice of $f: \pi_1(S^1 \times S^1) \to G$ and orientation $[S^1 \times S^1] \in H_2(S^1 \times S^1, \ZZ) \cong \ZZ$.
Therefore our preferred definition of the Bogomolov multiplier $B(G)$ is equivalent to the following: the quotient of the Schur multiplier $M(G)$ by the subgroup generated by all toral classes.

Recently, the first author noticed that certain invariants of branched $G$-covers studied in \cite{Brand,Sa20} (referred to as ``$G$-branched Schur invariants" in the latter) are valued in (torsors over) the Bogomolov multiplier \cite{Sa22}.
Using this perspective, he found counterexamples to an old conjecture in equivariant bordism, which had been revived by Uribe in \cite{Ur18} as the ``unitary evenness conjecture.''
The now-disproved conjecture states: the $G$-equivariant stable unitary bordism group $\Omega_*^{U,G}$ is a free $\Omega^U_*$-module on even-dimensional generators whenever $G$ is a compact Lie group.\footnote{
We note that Sophie Kriz has independently found counterexamples to the homotopical version of the conjecture \cite{Kriz} by identifying groups where $(MU_G)_*$ has support in odd degree.
As far as we are aware, it is an interesting open question whether there are any $G$ where $\Omega_*^{U,G}$ has such a property. 
}
The main result of \cite{Sa22} implies that for any odd order group $G$ with $B(G) \ne 0$, the 2-dimensional equivariant bordism group $\Omega_2^{U,G}$ has torsion.
The result also implies a uniform explanation of the earlier results of \cite{DS22}, which showed, via case-by-case cut-and-paste topological methods, that for all finite abelian, dihedral, symmetric and alternating groups $G$, every \emph{free} action of $G$ on an oriented surface is trivial in $\Omega_2^{SO,G}$.

Building on the method of Samperton's counterexamples, Angel, Samperton, Segovia and Uribe then showed in \cite{ASSU23} that for any finite group
\[ \Tor(\Omega_2^{SO,G}) = \Tor(\Omega_2^{U,G}) = \bigoplus_{(K) \le G} B(N_G(K)/K),\]
where the direct sum is taken over all conjugacy classes of subgroups of $G$.
In particular, \cite{ASSU23} removes the odd order assumption from the main result of \cite{Sa22}.

Given the central role that the Bogomolov multiplier has played in two far flung subfields---complex birational geometry in one case, torsion in equivariant bordism in the other---it is natural to wonder if there are any deeper connections between these two subjects that can be teased out.
In our minds, there remains an unexplained conceptual gap between these two roles of $B(G)$.

The proper way to understand this gap is presumably to combine a careful reading of \cite{Bo87} with sufficient expertise in equivariant bordism.
But we leave this for future efforts.
In this work, as a smaller step, we revisit two sets of vanishing results for $B(G)$ that to date have only been considered from the algebro-geometric setting, but not the low-dimensional topological/bordism theoretic.

In Section 2, we put \K's proof that $B(G) = 0$ for all finite simple groups $G$ into the context of the now-proved Ore conjecture \cite{Ore,LOST1,LOST2}, which has a direct topological interpretation that we review.

In Section 3, we show that $B(G)=0$ for all finite Coxeter groups $G$ (and indeed, all complex reflection groups).
We first explain how, from the birational geometry side, the Chevalley-Shephard-Todd theorem makes this ``obvious."
We then generalize the cut-and-paste topological methods of \cite{DS22} to give a direct ``topological" argument that builds on elementary calculations of $M(G)$ due to Ihara and Yokonuma \cite{IY65}.

\subsection*{Acknowledgments}
Google Gemini was used to help prepare Figures \ref{f:cut} and \ref{f:cap}.
Otherwise, no AI was used in the preparation of this manuscript.
We thank Marco Boggi and Bernardo Uribe for helpful conversations.

\section{Simple groups}
In this section we prove
\begin{thm}[\cite{K}]
\label{thm:K}
If $G$ is a finite simple group, then its Bogomolov multiplier $B(G)=0$ vanishes.
\end{thm}
\noindent In doing so our goal is to streamline the idea of \K's proof by placing it in the context of the Ore conjecture.
Admittedly, in the details our proof boils down to the same final calculations as \K;
however, we aim to explain the proof in a way that makes it feel ``inevitable" from the topological perspective (at least, given one of the headline results about finite simple groups).

To begin our retelling of this tale, we first recall the Ore conjecture, which is now a known theorem\footnote{We note that the proof depends on the classification of finite simple groups.}:

\begin{thm}[Ore conjecture \cite{Ore}, proved in \cite{LOST1}]
\label{thm:Ore}
If $G$ is a finite, non-abelian simple group, then every element of $G$ is a commutator.  That is, for all $g \in G$ there exist $x,y \in G$ such that $g = [x,y]=xyx^{-1}y^{-1}$. \qed
\end{thm}

Non-abelian simple groups are perfect---$G=[G,G]$---so every element of such a group is a \emph{product} of commutators.\footnote{In the sequel, we may occasionally say ``simple group" when really we ought to say ``non-abelian finite simple group."}
But as we warn our undergraduate math majors: not every element of the commutator subgroup is itself a commutator!
The interesting content of the Ore conjecture is that every element of a non-abelian simple group really does equal some commutator outright.

We can rephrase Theorem \ref{thm:Ore} in terms of bordism in the following elementary way.
For any perfect group $G=[G,G]$, we have 
\[ 0 = G_{ab} = G/[G,G] = H_1(G;\ZZ)=H_1(BG;\ZZ) = \Omega_1^{SO}(BG) ,\]
where $BG$ is the classifying space of $G$.\footnote{
We note that $H_1(BG;\ZZ)$ does not technically \emph{equal} $\Omega_1^{SO}(BG)$, but since there is a canonical isomorphism, we only slightly abuse notation.
Similarly, $H_2(BG;\ZZ) = \Omega_2^{SO}(BG)$ and $H_3(BG;\ZZ) = \Omega_3^{SO}(BG)$.
}
Thus, every loop $\gamma$ in $BG$ is the boundary of some oriented surface $S$ mapped into $BG$.
If $S$ is genus $g$ with one boundary component $\partial S \cong S^1$, we put our basepoint on the boundary, and write $\pi_1(S)$ as the free group on the usual generators $a_1,b_1,\dots,a_g,b_g$, then the boundary loop is
\[ \partial S = [a_1,b_1][a_2,b_2]\cdots[a_g,b_g].\]
Defining the \emph{commutator length} of $g \in G=[G,G]$ to be
\[ cl(g) \defeq \min\{g \in \mathbb{N} \mid \exists a_1,b_1,\dots,a_g,b_g \in G: g = \prod_{i=1}^g [a_i,b_i]\},\]
we then see 
\[ cl(g) = \min\{g \in \mathbb{N} \mid g \text{ bounds a genus $g$ surface in $BG$}\}.\]
Therefore, the Ore conjecture is equivalent to the fact that for a non-abelian simple group $G$, every loop in $BG$ bounds an oriented surface $S$ \emph{that has genus $1$}.

Given that the Bogomolov multiplier $B_0(G)$ is precisely the quotient of $M(G)=H_2(BG;\ZZ) = \Omega_2^{SO}(BG)$ by the subgroup generated by all toral---\emph{i.e.}\ ``genus 1"---classes the reader is presumably starting to suss out some connection to Theorem \ref{thm:K}.
As it turns out, we will ultimately want an extended version of the Ore conjecture for \emph{quasi-simple} groups in order to prove Theorem \ref{thm:K}.
However, to give a flavor of where we're headed, we first give a topological proof, using Theorem \ref{thm:Ore}, of the following warm-up

\begin{prop}[Warm-up to Theorem \ref{thm:K}]
\label{prop:genus3}
Every finite simple group $G$ has a Schur multiplier $M(G)=H_2(G;\mathbb{Z}) = \Omega_2^{SO}(BG)$ generated by surfaces of genus 3.
\end{prop}

Of course we would prove Theorem \ref{thm:K} if we could improve ``genus 3" in this proposition to ``genus 1".

\begin{proof}
Since $H_2(G;\ZZ) = H_2(BG; \ZZ)= \Omega_2^{SO}(BG)$, every homology class can be represented as $f_*[S]$ where $f: S \to BG$ is a continuous map from a closed, genus $g$ surface with orientation class $[S] \in H_2(S; \ZZ) \cong \ZZ$.
Since $BG = K(G,1)$ is an Eilenberg-MacLane space, the homotopy class of $f$ is determined by the induced group homomorphism
\[ f_\#: \pi_1(S) \to \pi_1(BG) = G.\]
Recall that the surface group has presentation
\[ \pi_1(S) = \langle \alpha_1,\beta_1,\dots,\alpha_g,\beta_g \mid \prod_{i=1}^g [\alpha_i,\beta_i] = 1\rangle.\]
Let $\{\delta_j\}_{j=1}^{g-1}$ be a collection of disjoint simple closed curves with $\delta_j$ in the same free homotopy class as
\[ \prod_{i=1}^{j} [\alpha_i,\beta_i] \in \pi_1(S)\]
as in Figure \ref{f:cut}.
Cutting $S$ along all of the $\delta_j$ yields a disjoint union of genus 1 surfaces
\[ S_1 \sqcup S_2 \sqcup \cdots \sqcup S_{g-2} \sqcup S_{g-1} \]
where $S_1$ and $S_{g-1}$ each have one boundary component and each of $S_2, \cdots, S_{g-2}$ has two boundary components.

\begin{figure}
\centering
\begin{tikzpicture}

    \colorlet{surfacecolor}{gray!15}
    \colorlet{cutcolor}{black}
    \colorlet{boundcolor}{black}

    \def\drawhole#1{
        \begin{scope}[shift={(#1, 0)}]
            \draw[thick] (-0.5, 0.1) to[bend right=50] (0.5, 0.1);
            \draw[thick] (-0.35, -0.05) to[bend left=45] (0.35, -0.05);
        \end{scope}
    }

    \begin{scope}[shift={(0, 3)}]
        \node[font=] at (5, 1) {$S$};

            
        \draw[thick] (-3.75, 1) -- (3.75, 1);
        \draw[thick] (-3.75, -1) -- (3.75, -1);
        \draw[thick] (-3.75, 1) arc (90:270:1);
        \draw[thick] (3.75, 1) arc (90:-90:1);

        \drawhole{-3.75}
        \drawhole{-1.25}
        \drawhole{1.25}
        \drawhole{3.75}

        \foreach \x in {-2.5, 0, 2.5} {
            \draw[very thick, cutcolor, dashed] (\x, 1) arc (90:270:0.2 and 1); 
            \draw[very thick, cutcolor] (\x, -1) arc (-90:90:0.2 and 1);        
        }
        \node at (-2.5,-1.4) {$\delta_1$};
        \node at (0,-1.4) {$\cdots$};
        \node at (2.5,-1.4) {$\delta_{g-1}$};
    \end{scope}

    \draw[-{Stealth[scale=1]}, line width=1.5pt, darkgray] 
        (0, 1.3) -- (0, .3);


    \begin{scope}[shift={(-5.25, -1)}]
        \draw[thick] (0, 1) -- (1.25, 1);
        \draw[thick] (0, -1) -- (1.25, -1);
        \draw[thick] (0, 1) arc (90:270:1); 
        \draw[very thick, boundcolor, dashed] (1.25, 1) arc (90:270:0.2 and 1);
        \draw[very thick, boundcolor] (1.25, -1) arc (-90:90:0.2 and 1);
        \drawhole{0}
        \node[align=center] at (0, -1.5) {$S_1$};
    \end{scope}

    \begin{scope}[shift={(-1.75, -1)}]
        \fill[white] (-1.25, 0) ellipse (0.2 and 1);
        \draw[thick] (-1.25, 1) -- (1.25, 1);
        \draw[thick] (-1.25, -1) -- (1.25, -1);
        \draw[very thick, boundcolor] (-1.25, 0) ellipse (0.2 and 1);
        \draw[very thick, boundcolor, dashed] (1.25, 1) arc (90:270:0.2 and 1);
        \draw[very thick, boundcolor] (1.25, -1) arc (-90:90:0.2 and 1);
        \drawhole{0}
        \node[align=center] at (0, -1.5) {$S_2$};
    \end{scope}

    \begin{scope}[shift={(1.75, -1)}]
        \fill[white] (-1.25, 0) ellipse (0.2 and 1);
        \draw[thick] (-1.25, 1) -- (1.25, 1);
        \draw[thick] (-1.25, -1) -- (1.25, -1);
        \draw[very thick, boundcolor] (-1.25, 0) ellipse (0.2 and 1);
        \draw[very thick, boundcolor, dashed] (1.25, 1) arc (90:270:0.2 and 1);
        \draw[very thick, boundcolor] (1.25, -1) arc (-90:90:0.2 and 1);
        \drawhole{0}
        \node[align=center] at (0, -1.5) {$\cdots$};
    \end{scope}

    \begin{scope}[shift={(5.25, -1)}]
        \fill[white] (-1.25, 0) ellipse (0.2 and 1);
        \draw[thick] (-1.25, 1) -- (0, 1);
        \draw[thick] (-1.25, -1) -- (0, -1);
        \draw[thick] (0, 1) arc (90:-90:1); 
        \draw[very thick, boundcolor] (-1.25, 0) ellipse (0.2 and 1);
        \drawhole{0}
        \node[align=center] at (0, -1.5) {$S_g$};
    \end{scope}
\end{tikzpicture}
\caption{Cutting a genus $g$ surface $S$ along the curves $\delta_j$ yields $g-1$ genus 1 surfaces with boundary.}
\label{f:cut}
\end{figure}
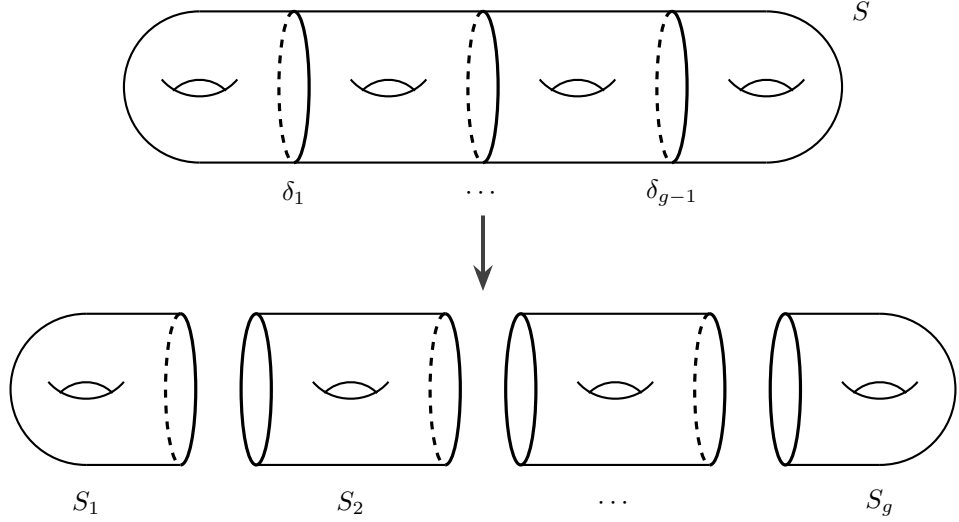

\begin{figure}
\centering
\begin{tikzpicture}[scale=.9]
    \colorlet{surfacecolor}{gray!15}
    \colorlet{cutcolor}{black}
    \colorlet{boundcolor}{black}

    \def\drawhole#1{
        \begin{scope}[shift={(#1, 0)}]
            \draw[thick] (-0.5, 0.1) to[bend right=50] (0.5, 0.1);
            \draw[thick] (-0.35, -0.05) to[bend left=45] (0.35, -0.05);
        \end{scope}
    }
    \begin{scope}[shift={(-5.25, -1)}]
        \draw[thick] (0, 1) -- (1.25, 1);
        \draw[thick] (0, -1) -- (1.25, -1);
        \draw[thick] (0, 1) arc (90:270:1); 
        \draw[very thick, boundcolor, dashed] (1.25, 1) arc (90:270:0.2 and 1);
        \draw[very thick, boundcolor] (1.25, -1) arc (-90:90:0.2 and 1);
        \drawhole{0}
        \node[align=center] at (-2, 0) {$S_1$};
        \draw[-{Stealth[scale=1]}, line width=1.5pt, darkgray] 
        (2,0) -- (3, 0);
    \end{scope}
        \begin{scope}[shift={(-1, -1)}]
        \draw[thick] (0, 1) -- (1.25, 1);
        \draw[thick] (0, -1) -- (1.25, -1);
        \draw[thick] (0, 1) arc (90:270:1); 
        \draw[very thick, boundcolor, dashed] (1.25, 1) arc (90:270:0.2 and 1);
        \draw[very thick, boundcolor] (1.25, -1) arc (-90:90:0.2 and 1);
        \drawhole{0}
        \node[align=center] at (3.5, 1) {$\hat{S}_1$};
    \end{scope}
        \begin{scope}[shift={(1.5, -1)}]
        \draw[thick] (-1.25, 1) -- (0, 1);
        \draw[thick] (-1.25, -1) -- (0, -1);
        \draw[thick] (0, 1) arc (90:-90:1); 
        \drawhole{0}
        \end{scope}

    \begin{scope}[shift={(-5.25, -4)}]
        \fill[white] (-1.25, 0) ellipse (0.2 and 1);
        \draw[thick] (-1.25, 1) -- (1.25, 1);
        \draw[thick] (-1.25, -1) -- (1.25, -1);
        \draw[very thick, boundcolor] (-1.25, 0) ellipse (0.2 and 1);
        \draw[very thick, boundcolor, dashed] (1.25, 1) arc (90:270:0.2 and 1);
        \draw[very thick, boundcolor] (1.25, -1) arc (-90:90:0.2 and 1);
        \drawhole{0}
        \node[align=center] at (-2, 0) {$S_j$};
        \draw[-{Stealth[scale=1]}, line width=1.5pt, darkgray] 
        (2,0) -- (3, 0);
    \end{scope}

 \begin{scope}[shift={(-1, -4)}]
        \draw[thick] (0, 1) -- (1.25, 1);
        \draw[thick] (0, -1) -- (1.25, -1);
        \draw[thick] (0, 1) arc (90:270:1); 
        \draw[very thick, boundcolor, dashed] (1.25, 1) arc (90:270:0.2 and 1);
        \draw[very thick, boundcolor] (1.25, -1) arc (-90:90:0.2 and 1);
        \drawhole{0}
        \node[align=center] at (6, 1) {$\hat{S}_j$};
    \end{scope}
    \begin{scope}[shift={(1.5, -4)}]
        \fill[white] (-1.25, 0) ellipse (0.2 and 1);
        \draw[thick] (-1.25, 1) -- (1.25, 1);
        \draw[thick] (-1.25, -1) -- (1.25, -1);
        \draw[very thick, boundcolor,dashed] (-1.25, 0) ellipse (0.2 and 1);
        \draw[very thick, boundcolor] (-1.25, -1) arc (-90:90:0.2 and 1);
        \draw[very thick, boundcolor, dashed] (1.25, 1) arc (90:270:0.2 and 1);
        \draw[very thick, boundcolor] (1.25, -1) arc (-90:90:0.2 and 1);
        \drawhole{0}
    \end{scope}
        \begin{scope}[shift={(4, -4)}]
        \draw[thick] (-1.25, 1) -- (0, 1);
        \draw[thick] (-1.25, -1) -- (0, -1);
        \draw[thick] (0, 1) arc (90:-90:1); 
        \drawhole{0}
        \end{scope}

        \begin{scope}[shift={(-5.25, -7)}]
        \fill[white] (-1.25, 0) ellipse (0.2 and 1);
        \draw[thick] (-1.25, 1) -- (0, 1);
        \draw[thick] (-1.25, -1) -- (0, -1);
        \draw[thick] (0, 1) arc (90:-90:1); 
        \draw[very thick, boundcolor] (-1.25, 0) ellipse (0.2 and 1);
        \drawhole{0}
        \node[align=center] at (-2, 0) {$S_g$};
        \draw[-{Stealth[scale=1]}, line width=1.5pt, darkgray] 
        (2,0) -- (3, 0);
    \end{scope}
            \begin{scope}[shift={(-1, -7)}]
        \draw[thick] (0, 1) -- (1.25, 1);
        \draw[thick] (0, -1) -- (1.25, -1);
        \draw[thick] (0, 1) arc (90:270:1); 
        \draw[very thick, boundcolor, dashed] (1.25, 1) arc (90:270:0.2 and 1);
        \draw[very thick, boundcolor] (1.25, -1) arc (-90:90:0.2 and 1);
        \drawhole{0}
        \node[align=center] at (3.5, 1) {$\hat{S}_g$};
    \end{scope}
        \begin{scope}[shift={(1.5, -7)}]
        \draw[thick] (-1.25, 1) -- (0, 1);
        \draw[thick] (-1.25, -1) -- (0, -1);
        \draw[thick] (0, 1) arc (90:-90:1); 
        \drawhole{0}
        \end{scope}
    
\end{tikzpicture}
\caption{Capping all boundary components of the genus 1 surfaces $S_j$ results in several surfaces of genus 2 and 3.}
\label{f:cap}
\end{figure}
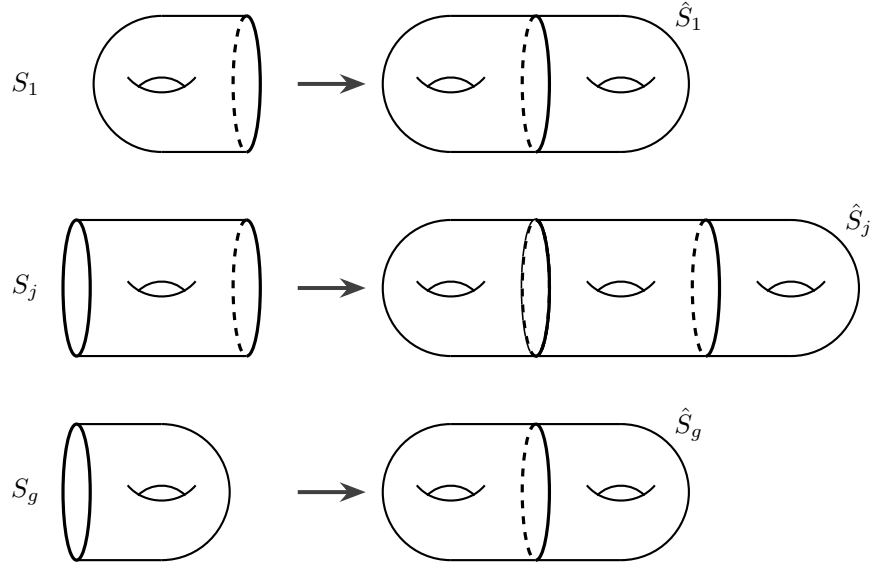

Each $S_j$ inherits a homomorphism
\[ f_j: \pi_1(S_j) \to G.\]
For $j=1,\dots,g-1$, let $\hat{S}_j$ be the surface formed by capping off each boundary component of $S_j$ by a \emph{genus 1 surface}; see Figure \ref{f:cap}.
Theorem \ref{thm:Ore} implies that $f_j$ extends to a homomorphism
\[ \hat{f}_j: \pi_1(\hat{S}_j) \to G.\]
By construction, there exists a 3-dimensional cobordism from $f: S \to BG$ to
\[ \bigsqcup_{j=1}^{g-1} \hat{f}_j: \hat{S}_j \to BG. \]
For $j=2,\dots,g-2$, each $\hat{S}_j$ has genus 3.
For $j=1$ and $g-1$, we can stabilize each of the genus 2 surfaces $\hat{S}_j$ with an additional 1-handle that maps trivially to $BG$.
This shows that $[f]$ is a sum of genus 3 classes, as needed.
\end{proof}

From our perspective, the reason the Ore conjecture only gets us down to genus 3 generators for $M(G)$ is that it is saying something about \emph{boundaries} of surfaces in $BG$ rather than \emph{closed} surfaces.
Indeed, Theorem \ref{thm:Ore} tells us exactly that boundaries of genus 1 surfaces are sufficient to exhibit every 1-cycle in $BG$ as a boundary.
But of course what we need to show in order to prove Theorem \ref{thm:K} is that \emph{closed genus 1 surfaces} are sufficient to generate \emph{every 2-cycle}.
To reformulate this kind of property as a group-theoretic statement about commutators, one place to look is in a Schur cover of $G$.

Recall that a group extension
\[ 1 \to K \to C \stackrel{\pi}{\to} G \to 1 \]
is called \emph{stem} if it is both central ($K \le Z(C)$) and has kernel contained in the commutator subgroup ($K \le [C,C]$).
If $G$ is a finite group, then a \emph{Schur cover} is a stem extension of $G$ of maximal order (one always exists); we call a group $C$ in such an extension a \emph{Schur covering group} of $G$.
For any Schur cover of $G$, it turns out that $K = \ker \pi$ is naturally isomorphic to the Schur multiplier $M(G)$.
For perfect $G$, there is a unique isomorphism class of Schur covering groups $C$.
More generally, any two Schur covering groups of a finite group $G$ are isoclinic.

The next lemma shows in what sense a Schur cover is a natural place to test for whether or not a group has a non-trivial Bogomolov multiplier.

\begin{lem}
\label{lem:cover}
A finite group $G$ has trivial Bogomolov  $B(G)=0$ if and only if it has a Schur cover $\pi: C \to G$ whose kernel $K = \ker \pi = M(G)$ is generated by commutators of $C$.  Moreover, if this is true of any Schur cover of $G$, then it is true for all of them.
\end{lem}

Two remarks are in order before we prove the lemma.

First, the lemma is presumably known to \K: see \cite[p.~214]{K}, where he more-or-less states it (but seemingly adds an unneeded assumption of quasi-simplicity) .
The contrapositive of our lemma is essentially the Pontryagin dual of \cite[Lem.~2.4]{BMP} (which \K \ also cites).

Second, this lemma gives some hint of how subtle it is to find groups with non-trivial Bogomolov multiplier.  The kernel $K$ of a Schur cover $\pi: C \to G$ is, by definition, always contained in the intersection of the center $Z(C)$ and the commutator subgroup $[C,C]$.  While it is perhaps natural to expect this intersection to be generated by commutators (since the commutator subgroup itself is), this is too naive and there is no reason it must be so.  In fact, counterexamples exist even for a few quasi-simple $G$, \emph{cf.}\ \cite{K}.

\begin{proof}
The first sentence of the lemma follows from the definition of $B(G)$ (\emph{i.e.}, the quotient of $M(G)$ by the subgroup generated by all toral classes) together with the following key point: toral classes of $M(G)$ are precisely the same thing as elements of $K$ that are commutators in $C$.
Let's sketch the proof of this latter fact, since it is the heart of the matter.

Let $f: S^1 \times S^1 \to BG$ represent the toral class $[f] = f_*[S^1\times S^1] \in M(G)=H_2(BG;\mathbb{Z})$.
As in the proof of Proposition \ref{prop:genus3}, the homotopy class of this map (hence, the homology class) is entirely specified by the monodromy homomorphism
\[ f_\#: \pi_1(S^1 \times S^1) \to \pi_1(BG)=G,\]
and, conversely, any such homomorphism determines a corresponding homology class.
Of course $\pi_1(S^1 \times S^1) \cong \ZZ \times \ZZ$ is the free abelian group of rank 2, so $f_\#$ is determined precisely by a pair of elements $a,b \in G$ that commute.
Let $\tilde{a},\tilde{b} \in C$ be any choice of lifts of $a,b$ to the cover.
Then (as can be argued either via the algebraic or topological definitions of $M(G)=H_2(G,\ZZ)=H_2(BG, \ZZ)$, whichever one finds more convenient)
\[ [\tilde{a},\tilde{b}] = [f] \in K = M(G).\]
Reversing this argument, we see that any commutator in $K$ arises in this form.

Moreover, notice that in either direction of this argument, it was immaterial which Schur cover we chose.
This observation implies the second sentence of the lemma.
\end{proof}

Lemma \ref{lem:cover} tells us that improving the ``genus 3" in Proposition \ref{prop:genus3} to ``genus 1" inevitably means proving the following: every Schur cover of a finite simple group has a kernel generated by commutators.
Now Schur covers of finite simple groups are examples of a slightly more general class of groups called \emph{quasisimple groups}---that is, perfect central extensions of a simple groups.
So, to prove Theorem \ref{thm:K}, one is led to wonder to what extent Ore's conjecture might hold in all quasisimple groups.
Conveniently, this is completely understood, thanks to the follow-up work \cite{LOST2} by the same team who proved Theorem \ref{thm:Ore}.
We state the subset of their results just for Schur covers, in which case $Z(C) = K = \ker \pi = M(G)$.

\begin{thm}[Extension of Ore's conjecture to Schur covers of simple groups, proved in \cite{Blau,LOST2}]
\label{thm:qs}
If $C$ is a Schur cover of a finite simple group $G$, then, apart from the finite number of exceptions listed in Table \ref{t:ore}, every element $x \in C$ is a commutator. \qed
\end{thm}

\begin{table}
\label{t:ore}
\centering
\begin{tabular}{ccc|c|c|c}
 & group $G$ & & Schur mult.~$M(G)$ & $|x|$ for & $|x|$ for \\
LOST & Wiki & Blau & (primary decomp.) & $x \in Z(C)$ & $x \notin Z(C)$\\ \hline
$A_6$ & $A_6$ & $A_6$ & $2 \times 3$ & 6 & 15, 24 \\
$A_7$ & $A_7$ & $A_7$ & $2 \times 3$ & 6 & 15 \\
$L_3(4)$ & $A_2(4)$ & $PSL(3,4)$ &  $3 \times 4 \times 4$ & 4, 6, 12 & 12, 84\\
$U_4(3)$ & ${}^2A_3(3^2)$ & $PSU(4,3^2)$ & $3 \times 3 \times 4$ & 6, 12 & 6, 12\\
$M_{22}$ & $M_{22}$ & $M_{22}$ & $3 \times 4$ & 6, 12 & -\\
$Fi_{22}$ & $Fi_{22}$ & $Fi_{22}$ & $2 \times 3$ & 6 & -\\
${}^2E_6(2)$ & ${}^2E_6(2^2)$ & ${}^2E_6(2^2)$ & $2 \times 2 \times 3$ & 6 & -
\end{tabular}
\vspace{.2cm}
\caption{The orders $|x|$ for all exceptions $x$ to the rule that elements of Schur covers of simple groups are equal to commutators.  
For convenience, we list the group names according to each of the conventions of LOST \cite{LOST2}, Wikipedia \cite{Wiki}, and Blau \cite{Blau}.
Note that $Z(C)$ and $M(G)$ are naturally isomorphic for $C$ a Schur cover of a simple group $G$.}
\end{table}

We are now ready to complete the proof of Theorem \ref{thm:K}.
The astute reader will notice that we do not need the results of the last column of Table \ref{t:ore}, which are the new contribution of \cite{LOST2}; the rest of the table was calculated by Blau much earlier \cite{Blau}.

\begin{proof}[Proof of Theorem \ref{thm:K}]
Theorem \ref{thm:qs} combines with Lemma \ref{lem:cover} to imply immediately that there is only a \emph{finite list} of finite simple groups whose Bogomolov multipliers we must argue vanish.

One can start by being a complete brute, and just try to ask, say, GAP \cite{GAP} to compute $B(G)$ for each of the seven groups $G$ in Table \ref{t:ore}.
In fact, the GAP package HAP \cite{HAP} has a built-in function BogomolovMultiplier, which can quickly confirm (on a modern laptop) vanishing for the four groups 
$A_6$, $A_7$, $L_3(4)$ and $M_{22}$, the largest of which is order
\[ |M_{22}| = 443,520. \]
The remaining three groups have orders
\[
\begin{aligned}
|Fi_{22}|&= 64,561,751,654,400 \\
|U_4(3)|&=101,798,586,432,000 \\
|{}^2E_6(2)|&= 76,532,479,683,774,853,939,200,
\end{aligned}
\]
and there is little chance for off-the-shelf, direct methods to work without unrealistic time or memory resources.
Alas, the brutish route only gets us so far.

To finish the calculations, one simply needs to look more carefully at the second and third columns of Table \ref{t:ore}.
For example, we see that the only elements $x \in H_2({}^2E_6(2^2), \ZZ)$ that are not toral have order 6.
But $H_2({}^2E_6(2^2), \ZZ) \cong \ZZ/2\ZZ \times \ZZ/2\ZZ \times \ZZ/3\ZZ$ is generated by elements of order 2 and 3.
Since these are all toral classes, we conclude that $B({}^2E_6(2^2))=0$.
In fact, the same kind of reasoning shows that $B(G) = 0$ for all seven of the groups in Table \ref{t:ore} \emph{except} for $L_3(4)$, but we already eliminated that one via direct calculation.
\end{proof}

It has been noted before that finite simple groups have---perhaps surprisingly---``small" Schur multipliers, in the sense that $|M(G)|$ is often quite small compared to $|G|$.
See MathOverflow for a nice discussion of this point, especially Jesper Grodal's answer to a question of Terence Tao \cite{MO}.
The fact that $B(G)=0$ adds further to the ``smallness": $M(G)=\Omega_2^{SO}(BG)$ is ``small" in the sense that genus 1 classes are sufficient to generate all of it.

\section{Finite Coxeter groups}
A \emph{pseudoreflection} on a finite-dimensional complex vector space $V$ is an invertible linear transformation $g: V \longrightarrow V$ of finite order such that the subspace of fixed points $V^g = \{ v \in V : gv = v \}$ has codimension one. A finite subgroup of $\GL(V)$ is called a \emph{complex reflection group} if it is generated by pseudoreflections. An important result at the intersection of combinatorics and geometry is the Chevalley--Shephard--Todd theorem.

\begin{thm}[Shephard--Todd \cite{ST54}, Chevalley \cite{Ch55}] Let $V$ be a finite dimensional vector space over the complex numbers $\CC$ and let $G$ be a finite subgroup of the general linear group $\operatorname{GL}(V)$. Then the following are equivalent:
\begin{enumerate}
    \item The group $G$ is a complex reflection group.
    \item The algebra of $G$-invariant polynomials $\CC[V]^G$ is a polynomial algebra. \qed
\end{enumerate}
\end{thm}

In particular, $\CC(V)^G = \CC(V/G)$ is a field of rational functions.
In other words, $V/G$ is a rational algebraic variety.
Thus, given the role of $B(G)$ explained in the introduction, this proves

\begin{cor}
    If $G$ is a complex reflection group, then its Bogomolov multiplier $B(G)=0$ vanishes.
\end{cor}

In the rest of this section, we will give a new, elementary proof of this fact in the special case of orthogonal reflections on a Euclidean space---\emph{i.e.,}\ finite Coxeter groups.  Our aim is to show that for such groups, their Schur multiplier has a generating set where all classes are toral. 

The Bogomolov multiplier satisfies pertinent properties for the product of two groups: $B(G_1\times G_2)\cong B(G_1)\times B(G_2)$.
A finite Coxeter group $G$ is characterized as admitting a presentation of the form
\ben
\langle r_1,\cdots, r_n | (r_ir_j)^{m_{ij}}=1,m_{ii}=1,m_{ij}=m_{ji}\in \{2,3,\cdots\},i\neq j,1\leq i,j\leq n\rangle\,.
\een
There is an associated Coxeter diagram $\Pi(G)$ which contains $n$ points corresponding to the generators $r_i$ ($1\leq i\leq n$) in a one-to-one way, and if $m_{ij}\neq 2$ ($i\neq j$), the two points corresponding to $r_i$ and $r_j$, respectively, are connected by segments together with the number $m_{ij}$. A Coxeter group is irreducible if and only if the diagram $\Pi(G)$ is connected. Assume $\Pi(G)$ is a disjoint union of two subdiagrams $\Pi_1$ and $\Pi_2$. If $G_1$ and $G_2$ are the associated Coxeter groups for the diagrams $\Pi_1$ and $\Pi_2$, then $G=G_1\times G_2$. A fundamental result in this area is the following.

\begin{thm}[\cite{Cox2,Wit41}] The following list contains all the irreducible finite Coxeter groups:\begin{tabbing}
\hspace{3cm}             \= \hspace{4cm}  \\
$A_n\,(n\geq 1):$          \> $\xymatrix{ \sta{r_1}{\ci}\ar@{-}[r]_3  & \sta{r_2}{\ci} \ar@{.}[r]& \stackrel{r_{n-1}}{\circ}\ar@{-}[r]_3 &\stackrel{r_{n}}{\circ}}$ \\
$B_n\,(n\geq 2):$           \>$\xymatrix{\sta{r_1}{\ci}\ar@{-}[r]_3  & \sta{r_2}{\ci} \ar@{.}[r]& \sta{r_{n-2}}{\ci}\ar@{-}[r]_3 &\sta{r_{n-1}}{\ci}\ar@{-}[r]_4&\sta{r_{n}}{\ci}}$\\
$D_n\,(n\geq 4):$           \>$\xymatrix{\sta{r_1}{\ci}\ar@{-}[r]_3  & \sta{r_2}{\ci} \ar@{.}[r]& \sta{r_{n-3}}{\ci}\ar@{-}[r]_3 &\sta{r_{n-2}}{\ci}\ar@{-}[r]_3\ar@{-}[d]^3&\sta{r_{n}}{\ci}&\\
&&&\underset{r_{n-1}}{\ci}}$\\
$E\,(n=6,7,8):$             \>$\xymatrix{\sta{r_1}{\ci}\ar@{-}[r]_3  & \sta{r_2}{\ci} \ar@{.}[r]& \sta{r_{n-4}}{\ci}\ar@{-}[r]_3 &\sta{r_{n-3}}{\ci}\ar@{-}[r]_3\ar@{-}[d]^3&\sta{r_{n-1}}{\ci}\ar@{-}[r]_3&\sta{r_{n}}{\ci}
\\
&&&\underset{r_{n-2}}{\ci}}$\\
$F_4:$                      \>$\xymatrix{\sta{r_1}{\ci}\ar@{-}[r]_3  & \sta{r_2}{\ci} \ar@{-}[r]_4& \stackrel{r_{3}}{\circ}\ar@{-}[r]_3 &\stackrel{r_{4}}{\circ}}$\\
$G_2^{(n)}(n\geq 5):$\>$\xymatrix{\sta{r_1}{\ci}\ar@{-}[r]_n  & \sta{r_2}{\ci}}$\\
$H_3:$\>$\xymatrix{\sta{r_1}{\ci}\ar@{-}[r]_3  & \sta{r_2}{\ci}\ar@{-}[r]_5&\sta{r_3}{\ci}}$\\
$H_4:$\>$\xymatrix{\sta{r_1}{\ci}\ar@{-}[r]_3  & \sta{r_2}{\ci}\ar@{-}[r]_3&\sta{r_3}{\ci}\ar@{-}[r]_5&\sta{r_4}{\ci}}$\\
\end{tabbing}
\end{thm}

Ihara and Yokonuma calculated in \cite{IY65} the Schur multiplier for the irreducible finite Coxeter groups using the cohomological interpretation of the Schur multiplier. A good account of the proof of this result can be found in the book by Karpilovsky \cite{Kar87}. 
\begin{thm}[Ihara-Yokonuma \cite{IY65}]\label{IY} The Schur multiplier $M(G)$ of an irreducible Coxeter group has the following values:
\begin{itemize}
    \item[(i)] If $G$ is of type $A_n$, then $M(G)\cong \left\{   
    \begin{array}{ll}
        1 &  \textrm{if }n\leq 2\\
        \ZZ_2 & \textrm{if }n\geq 3\,. 
    \end{array}\right.$
    \item[(ii)] If $G$ is of type $B_n$, then $M(G)\cong 
    \left\{\begin{array}{ll}
        \ZZ_2 &  \textrm{if }n = 2\\
        \ZZ_2\times \ZZ_2 & \textrm{if }n= 3\\
        \ZZ_2\times \ZZ_2\times\ZZ_2 & \textrm{if }n\geq 4\,.
    \end{array}\right.$
    \item[(iii)] If $G$ is of type $D_n$, then $M(G)\cong \left\{
    \begin{array}{ll}
        \ZZ_2\times \ZZ_2\times \ZZ_2 &  \textrm{if }n=4\\
        \ZZ_2\times \ZZ_2 & \textrm{if }n\geq 5\,. 
    \end{array}\right.$
    \item[(iv)] If $G$ is of type $E_6$, $E_7$, or $E_8$, then $M(G)\cong \ZZ_2$.
    \item[(v)] If $G$ is of type $F_4$, then $M(G)\cong \ZZ_2\times \ZZ_2$.
    \item[(vi)] If $G$ is of type $G_2^{(n)}$, then $M(G)\cong \left\{   
    \begin{array}{ll}
        1 &  \textrm{if }n\textrm{ is odd}\\
        \ZZ_2 & \textrm{if }n\textrm{ is even}\,. 
    \end{array}\right.$
    \item[(vii)] If $G$ is of type $H_3$ or $H_4$, then $M(G)\cong \ZZ_2$.
\end{itemize}
\end{thm}

Miller \cite{Mi52} provides a combinatorial interpretation of the Schur multiplier considering the free group on pairs $\langle x, y\rangle$ with $x,y\in G$. We denote this group by $\langle G,G\rangle$ and take the kernel of the canonical map to the commutator group $\langle G,G\rangle\longrightarrow [G,G]$. The Schur multiplier is isomorphic to the following quotient
\[M(G)\cong \frac{\op{ker}\left(\langle G,G\rangle\longrightarrow [G,G]\right)}{N}\,,\]
where $N$ is the normal subgroup generated by the following four relations: 
 \begin{eqnarray}\label{four1}
		&\left<x,x\right>&\sim 1\,,\\\label{four2}
	&\left<x,y\right>&\sim \left<y,x\right>^{-1}\,,\\\label{four3}
	&\left<xy,z\right>&\sim \left<y,z\right>^x\left<x,z\right>\,,\\\label{four4}
	&\left<y,z\right>^x&\sim \left<x,[y,z]\right>\left<y,z\right>\,,
\end{eqnarray}
Miller also derived several further relations from these, which we list in the following theorem.
\begin{thm}[\cite{Mi52}]
The following relations can be deduced from \eqref{four1}-\eqref{four4}:
 \begin{eqnarray}\label{four5}
 &\left<x,yz\right>&\sim \left<x,y\right>\left<x,z\right>^y\,,\\\label{four6}
 &\left<x,y\right>^{\left<a,b\right>}&\sim \left<x,y\right>^{[a,b]}\,,\\\label{four7}
 & \left[\left<x,y\right>,\left<a,b\right>\right]&\sim \left<[x,y],[a,b]\right>\,,\\\label{four8}
 & \left<b,b'\right>\left<a_0,b_0\right> &\sim \left<[b,b'],a_0\right>\left<a_0,[b,b']b_0\right>\left<b,b'\right>\,,\\\label{four9}
 & \left<b,b'\right>\left<a_0,b_0\right> &\sim \left<[b,b']b_0,a_0\right>\left<a_0,[b,b']\right>\left<b,b'\right>\,,\\\label{four10}
 & \left<b,b'\right>\left<a,a'\right> &\sim \left<[b,b'],[a,a']\right>\left<a,a'\right>\left<b,b'\right>\,,\\\label{four11}
 &\left<x^n,x^s\right> &\sim 1\hspace{1cm}n=0,\pm1,\cdots;s=0,\pm1,\cdots\,,
 \end{eqnarray}for $x,y,z,a,b,a',b',a_0,b_0\in G$.
\end{thm}

Each of the Miller relations has a direct cut-and-paste low-dimensional interpretation; see \cite[\S1.2]{DS22}.

Dom\'inguez-Segovia \cite{DS22} showed that the toral classes associated to the generators of the Schur multiplier for dihedral and symmetric groups are as follows: 
\begin{itemize}
    \item[(i)] For dihedral groups ($A_2$, $D_2$, $G_2^{(n)}$ ($n\geq 5$)) with presentation $$D_{2n}=\langle a,b:a^2=1, b^2=1, (ab)^n=1\rangle\,,$$ 
for $n$ odd, we have $M(D_{2n})=0$ and for $n$ even of the form $n=2k$, the generator of $M(D_{4k})=\ZZ_2$ is $\langle c^k,a\rangle$ with $c=ab$.
\item[(ii)] For symmetric groups ($A_n$ ($n\geq 1$)) with presentation $$S_n=\langle r_1,\cdots, r_{n-1}: r_i^2=1, (r_ir_{i+1})^3=1, r_ir_j=r_jr_i \text{ for } |i-j|>1\rangle\,,$$
in the case $n\leq 3$ we have $M(S_n)=0$ and for $n\geq 4$, the generator of $M(S_n)=\ZZ_2$ is represented by any pair $\langle r_i,r_j\rangle$ with $|i-j|>1$.
\end{itemize}
Similarly, for finite Coxeter groups, there are two types of generators:
\begin{itemize}
    \item[(a)]Dihedral generators: $\langle (r_ir_j)^k,r_i\rangle$ with $m_{ij}=2k$, $k\geq 2$, where we can show using \eqref{four1}-\eqref{four11} that 
    $\langle (r_ir_j)^k,r_i\rangle\sim \langle r_ir_j,r_i\rangle^k$. This element is denoted in the diagram as follows 
    \vspace{.3cm}
    $$\xymatrix{ \sta{r_i}{\ci}\ar@{-}[r]_{2k}\ar@/^1.5pc/@{~}[r]  & \sta{r_j}{\ci}}$$
    \item[(b)]Symmetric generators: $\langle r_i,r_j\rangle$ with $m_{ij}=2$, where we can show using \eqref{four4} that
    $\langle r_i,r_j\rangle\sim \langle r_i,r_j\rangle^r$ for any $r$. This element is denoted in the diagram as follows 
    
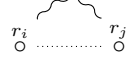
\vspace{.3cm}
    $$\xymatrix{ \sta{r_i}{\ci}\ar@{.}[r]\ar@/^1.5pc/@{~}[r]  & \sta{r_j}{\ci}}$$
\end{itemize}

In order to move among the generators we need the following lemma, which is analogous to \cite[Lem 7.2.7]{Kar87}.

\begin{lem}
\begin{itemize}
    \item[(i)] Assume that $r_i$, $r_j$ and $r_k$ are arranged in the following diagram:
    
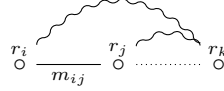
\vspace{.3cm}
    $$\xymatrix{ \sta{r_i}{\ci}\ar@{-}[r]_{m_{ij}} \ar@/^2pc/@{~}[rr]  & \sta{r_j}{\ci} \ar@{.}[r] \ar@/^1pc/@{~}[r]& \stackrel{r_{k}}{\circ}}$$
with $m_{ij}=2n_{ij}+1$. Then $\langle r_i,r_k\rangle\sim \langle r_j,r_k\rangle $.
\item[(ii)] Assume that $r_i$, $r_j$, $r_k$ and $r_l$ are arranged in the following diagram:
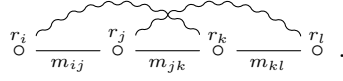
\vspace{.3cm}
$$\xymatrix{ \sta{r_i}{\ci}\ar@{-}[r]_{m_{ij}}\ar@/^1.5pc/@{~}[rr]  & \sta{r_j}{\ci} \ar@{-}[r]_{m_{jk}}\ar@/^1.5pc/@{~}[rr]& \stackrel{r_{k}}{\circ}\ar@{-}[r]_{m_{kl}} &\stackrel{r_{l}}{\circ}}\,.$$
with $m_{ij}=2n_{ij}+1$ and $m_{kl}=2n_{kl}+1$. Then $\langle r_i,r_k\rangle\sim \langle r_j,r_l\rangle $.
\end{itemize}
\end{lem}
\begin{proof}
For $(i)$, we have the hypothesis that $r_k$ commutes with $r_i$ and $r_j$. Set $r=(r_ir_j)^{n_{ij}}$, hence $\langle r_i,r_k \rangle\sim \langle r_i,r_k\rangle^r=\langle r_j,r_k\rangle$. For $(ii)$, we have the hypothesis that $r_i$ commutes with $r_l$ and $r_k$, and $r_l$ commutes with $r_i$ and $r_j$. Moreover, set $r=(r_ir_j)^{n_{ij}}(r_kr_l)^{n_{kl}}$, hence $rr_ir^{-1}=r_j$ and $rr_kr^{-1}=r_l$, so $\langle r_i,r_k \rangle\sim \langle r_i,r_k\rangle^r=\langle r_j,r_l\rangle$.
\end{proof}

For the list of irreducible finite Coxeter groups provided in Theorem \ref{IY}, we write down the generators for the Schur multiplier specified in Theorem \ref{IY}.

\begin{itemize}
    \item[(i)] These are precisely the symmetric groups.
    \item[(ii)] $B_2$ is the dihedral group $D_4$, and $B_3$ has the generators 
    
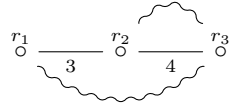
\vspace{.3cm}
    $$\xymatrix{ \sta{r_1}{\ci}\ar@{-}[r]_{3}\ar@/_1.5pc/@{~}[rr]  & \sta{r_2}{\ci} \ar@{-}[r]_{4}\ar@/^1.5pc/@{~}[r]& \stackrel{r_{3}}{\circ}}\,$$

    \vspace{.3cm}
    
    \noindent given by $\langle r_1,r_3\rangle$ and $\langle (r_2r_3)^2,r_2\rangle$. $B_4$ has the generators 
    
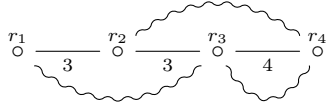
\vspace{.3cm}
    $$\xymatrix{ \sta{r_1}{\ci}\ar@{-}[r]_{3}\ar@/_1.5pc/@{~}[rr]  & \sta{r_2}{\ci} \ar@{-}[r]_{3}\ar@/^1.5pc/@{~}[rr]& \stackrel{r_{3}}{\circ}\ar@/_1.5pc/@{~}[r]\ar@{-}[r]_{4}&\stackrel{r_{4}}{\circ}}\,$$

    \vspace{.3cm}
    
    \noindent given by $\langle r_1,r_3\rangle$, $\langle r_2,r_4\rangle$ and $\langle (r_3r_4)^2,r_3\rangle$.
    \item[(iii)] $D_4$ has the generators 
    
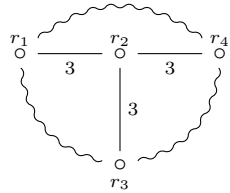
\vspace{.3cm}
    $$\xymatrix{\sta{r_1}{\ci}\ar@{-}[r]_3\ar@/^1.5pc/@{~}[rr] \ar@/_1pc/@{~}[rd]&\sta{r_2}{\ci}\ar@{-}[r]_3\ar@{-}[d]^3&\sta{r_4}{\ci}\\
&\underset{r_{3}}{\ci}\ar@/_1pc/@{~}[ru]&}$$
 given by $\langle r_1,r_4\rangle$, $\langle r_1,r_3\rangle$ and $\langle r_3,r_4\rangle$. $D_5$ has the generators
 
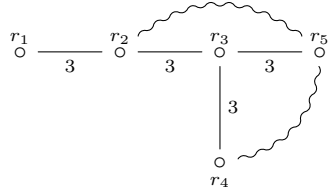
\vspace{.3cm}
  $$\xymatrix{\sta{r_1}{\ci}\ar@{-}[r]_3&\sta{r_2}{\ci}\ar@{-}[r]_3\ar@/^1.5pc/@{~}[rr] &\sta{r_3}{\ci}\ar@{-}[r]_3\ar@{-}[d]^3&\sta{r_5}{\ci}\\
&&\underset{r_{4}}{\ci}\ar@/_1pc/@{~}[ru]&}$$
given by $\langle r_2,r_5\rangle$ and $\langle r_4,r_5\rangle$. Here the freedom of an additional arc with weight 3 allows the identification of two of the generators in the case of the group $D_4$.
 \item[(iv)] $E_6$, $E_7$ and $E_8$ each have only one generator.
 \item[(v)] $F_4$ has the generators
 
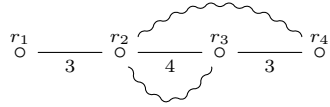
\vspace{.3cm}
  $$\xymatrix{\sta{r_1}{\ci}\ar@{-}[r]_3&\sta{r_2}{\ci}\ar@/_1.5pc/@{~}[r]\ar@{-}[r]_4\ar@/^1.5pc/@{~}[rr]&\sta{r_3}{\ci}\ar@{-}[r]_3&\sta{r_4}{\ci}}$$

    \vspace{.3cm}
    
    \noindent given by $\langle r_2,r_4\rangle$ and $\langle (r_2r_3)^2,r_2\rangle$. 
\item[(vi)] These are the dihedral groups.
\item[(vii)] $H_3$ has the generator
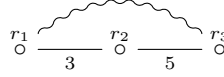
\vspace{.3cm}
  $$\xymatrix{\sta{r_1}{\ci}\ar@{-}[r]_3\ar@/^1.5pc/@{~}[rr]&\sta{r_2}{\ci}\ar@{-}[r]_5&\sta{r_3}{\ci}}$$
given by $\langle r_1,r_3\rangle$.
\item[(viii)] $H_4$ has the generator
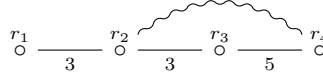
\vspace{.3cm}
  $$\xymatrix{\sta{r_1}{\ci}\ar@{-}[r]_3&\sta{r_2}{\ci}\ar@{-}[r]_3\ar@/^1.5pc/@{~}[rr]&
  \sta{r_3}{\ci}\ar@{-}[r]_5&\sta{r_4}{\ci}}$$
  given by $\langle r_2,r_4\rangle$.
\end{itemize}
We have obtained all the generators of the Schur multiplier associated to the list of irreducible finite Coxeter groups in Theorem \ref{IY} and, moreover, we have proved that the generators are represented by toral classes. Thus the Bogomolov multiplier $B(G)=0$ vanishes for any finite Coxeter group $G$.

\bibliographystyle{amsalpha}
\bibliography{refs}

\end{document}